\documentclass{article}
\usepackage{amsmath,amssymb,amsthm,color,graphicx,diagbox,tikz,colortbl,array}
\usepackage{setspace}
\usepackage{float}

\title{Provability interpretation of non-normal modal logics having neighborhood semantics}
\author{Haruka Kogure\footnote{Email: kogure@stu.kobe-u.ac.jp}
\footnote{Graduate School of System Informatics, Kobe University, 1-1 Rokkodai, Nada, Kobe 657-8501, Japan.}}

\date{}

\theoremstyle{plain}
\newtheorem{thm}{Theorem}[section]
\newtheorem*{thm*}{Theorem}

\newtheorem{prop}[thm]{Proposition}
\newtheorem{cor}[thm]{Corollary}

\newtheorem*{fact*}{Fact}
\newtheorem{prob}[thm]{Problem}
\newtheorem*{prob*}{Problem}
\newtheorem*{cl*}{Claim}
\newtheorem{cl}{Claim}[section]
\newtheorem*{scl*}{Subclaim}

\theoremstyle{definition}
\newtheorem{defn}[thm]{Definition}

\newcommand{\PL}{\mathsf{PL}}

\newcommand{\MF}{\mathsf{MF}}
\newcommand{\N}{\mathsf{N}}
\newcommand{\tc}{\vdash^{\mathrm{t}}}

\newcommand{\PA}{\mathsf{PA}}
\newcommand{\PR}{\mathrm{Pr}}

\newcommand{\Proof}{\mathrm{Proof}}
\newcommand{\Prov}{\mathrm{Prov}}
\newcommand{\Con}{\mathrm{Con}}

\newcommand{\gn}[1]{\ulcorner#1\urcorner}

\newcommand{\D}[1]{\mathbf{D#1}}

\newcommand{\num}{\overline}

\newcommand{\LA}{\mathcal{L}_A}

\newcommand{\EN}{\mathsf{EN}}
\newcommand{\ECN}{\mathsf{ECN}}
\newcommand{\ENP}{\mathsf{ENP}}
\newcommand{\END}{\mathsf{END}}
\newcommand{\ECNP}{\mathsf{ECNP}}

\newcommand{\GL}{\mathsf{GL}}

\begin{document}

\maketitle
\begin{abstract}
We study provability predicates $\mathrm{Pr}_T(x)$ satisfying the following  condition $\mathbf{E}$ from a modal logical perspective:
\\
$\mathbf{E}:$ if $ 
T \vdash \varphi \leftrightarrow \psi$, then $T \vdash \mathrm{Pr}_T(\ulcorner \varphi \urcorner) \leftrightarrow \mathrm{Pr}_T(\ulcorner \psi \urcorner)$.\\
For this purpose, we develop a new method of embedding models based on neighborhood semantics into arithmetic.
Our method broadens the scope of arithmetical completeness proofs.
In particular, we prove the arithmetical completeness theorems for the non-normal modal logics $\mathsf{EN}$, $\mathsf{ECN}$, $\mathsf{ENP}$, $\mathsf{END}$, and $\mathsf{ECNP}$.

% by embedding their neighborhood semantics into arithmetic. This embedding establishes a method that extends Solovay’s technique for proving arithmetical completeness to the framework of neighborhood semantics.
\end{abstract}
\section{Introduction}
Let $T$ denote a consistent primitive recursively axiomatized theory containing Peano arithmetic $\PA$.
In the usual proof of the second incompleteness theorem, a provability predicate $\PR_T(x)$ for $T$ plays a central role. The well-known second incompleteness theorem states that
if $\PR_T(x)$ satisfies Hilbert--Bernays--L\"{o}b's derivability conditions:
\begin{align*}
\D{2}: & \ T \vdash \PR_T(\gn{\varphi \to \psi}) \to (\PR_T(\gn{\varphi}) \to \PR_T(\gn{\psi})), \\
\D{3}: & \ T \vdash \PR_T(\gn{\varphi}) \to \PR_T(\gn{\PR_T(\gn{\varphi})}), 
\end{align*}
 then $T \nvdash \neg \PR_T(\gn{0=1})$ holds.
Furthermore, Kurahashi \cite{Kur25} established a refined version of the second incompleteness theorem,
 showing that the result holds under weaker assumptions on provability predicates.
In particular, he introduced the following conditions, which are weaker than $\D{2}$:
\begin{align*}
\mathbf{E}: & \ T \vdash \varphi \leftrightarrow \psi \Longrightarrow T \vdash \PR_T(\gn{\varphi}) \leftrightarrow \PR_T(\gn{\psi}), \\
\mathbf{C}: & \ T \vdash \PR_T(\gn{\varphi}) \wedge \PR_T(\gn{\psi}) \to \PR_T(\gn{\varphi \wedge \psi}),
\end{align*}
and showed that if $\PR_T(x)$ satisfies $\mathbf{E}$, $\mathbf{C}$, and $\D{3}$, then $T \nvdash \neg \PR_T(\gn{0=1})$ holds.
Moreover, the conditions 
$\mathbf{E}$ and $\mathbf{C}$ are also sufficient for the second incompleteness theorem based on other formulations of consistency.
% Also, in addition to $\neg \PR_T(\gn{0 = 1})$, many alternative formulations of  consistency statements have been proposed in the context of the second incompleteness theorem. The conditions $\mathbf{E}$ and $\mathbf{C}$ play a central role in the second incompleteness theorems based on alternative consistency statements.
 Kurahashi \cite{Kur20-2} introduced the consistency statement $\Con^S_T := \{ \neg (\PR_T(\gn{\varphi}) \wedge \PR_T(\gn{\neg \varphi})) \mid \varphi \text{ is a formula} \}$, which is called the schematic consistency 
statement. He also introduced 
the following consistency statement $\mathbf{Ros}: T \vdash \neg \varphi \Longrightarrow T \vdash \neg \PR_T(\gn{\varphi})$.
He proved the following:
\begin{itemize}
    \item 
     If $\PR_T(x)$ satisfies $\mathbf{E}$ and $\D{3}$, then $T \nvdash \Con^S_T$.
     \item
    If $\PR_T(x)$ satisfies $\mathbf{C}$ and $\D{3}$, then $\mathbf{Ros}$ does not hold for $\PR_T(x)$.
    
\end{itemize}

One of the most important result in provability logic
is Solovay’s arithmetical completeness theorem \cite{Sol}.
Solovay’s theorem states that for any $\Sigma_{1}$-sound recursively enumerable extension $T$ of $\PA$, modal logical principles verifiable in $T$ of the standard provability predicate $\Prov_T(x)$ are exactly characterized by the normal modal logic $\GL$. 
His proof is carried out by embedding suitable finite Kripke models for $\GL$ into arithmetic.
In our previous work \cite{KK2}, we investigated provability predicates $\PR_T(x)$ satisfying the condition $\mathbf{M} : T \vdash \varphi \to \psi \Longrightarrow T \vdash \PR_T(\gn{\varphi}) \to \PR_T(\gn{\psi})$.
We proved the arithmetical completeness theorems for the monotonic modal logic $\mathsf{MN}$, which is the logic equipped with the rule $(\mathsf{RM})\dfrac{A \to B}{\Box A \to \Box B}$.
Although $\mathsf{MN}$ is a non-normal modal logic that does not admit a Kripke semantics, it is known to have a monotonic neighborhood semantics (\cite{Che,Eric}).
In order to apply Solovay's method, we introduced a  relational semantics, which is similar to Kripke semantics, for $\mathsf{MN}$ and established its arithmetical completeness theorem by embedding this semantics into arithmetic.
Also, Fitting, Marek, and Truszczy\'{n}ski \cite{fmt}  introduced the non-normal modal logic $\mathsf{N}$, known as the pure logic of necessitation, and provided a Kripke-like relational semantics. 
Kurahashi \cite{Kur23} and the author \cite{kog} proved the arithmetical completeness theorems for several extensions of $\N$ by embedding such relational semantics into arithmetic.

Our argument on the proof of the arithmetical completeness for $\mathsf{MN}$ relies on relational semantics.
This naturally leads to the question of whether our framework can be extended to neighborhood semantics in general.
The modal logic $\mathsf{EN}$, which is weaker than $\mathsf{MN}$, is equipped with the rule  $(\mathsf{RE})\dfrac{A \leftrightarrow B}{\Box A \leftrightarrow \Box B}$ and 
the logic $\mathsf{ECN}$ is obtained from $\EN$ by adding the axiom $\mathsf{C}: \Box A \wedge \Box B \to \Box (A \wedge B)$. 
They have been traditionally studied as non-normal modal logics having neighborhood semantics (see \cite{Che, Eric}).
%Our argument on MN rely on relational semantics
%neighborhood semantics in genera
In \cite{KK2}, we posed the following question:
\begin{prob}[{\cite[Problem 7.1]{KK2}}]
	Is the logic $\mathsf{EN}$ arithmetically complete with respect to provability predicates $\PR_T(x)$ satisfying
the condition $\mathbf{E}$?
	\end{prob}
In this paper, we answer the question affirmatively and prove the arithmetical completeness theorems for $\EN$ and $\ECN$.
Unlike previous approaches to arithmetical completeness theorems that rely on relational semantics, our proofs directly embed neighborhood semantics into arithmetic, thereby extending Solovay’s technique to the setting of neighborhood semantics. 
We further introduce the systems
$\mathsf{ENP} = \EN + \neg \Box \bot$,
$\END = \EN + \neg (\Box A \wedge \Box \neg A)$, and
$\ECNP = \ECN + \neg \Box \bot$,
where the axioms $\neg \Box \bot$ and $\neg (\Box A \wedge \Box \neg A)$ correspond to the consistency statements $\neg \PR_T(\gn{0=1})$ and $\Con^S_T$, respectively.
We then prove the arithmetical completeness theorems for these systems.

%In Appendix, we construct a 

This paper is organized as follows. In Section \ref{pre}, we provide preliminaries and background of the second incompleteness theorem and non-normal modal logics having neighborhood semantics.
Sections \ref{pr1} and \ref{pr2} present the proofs of arithmetical completeness theorems for $\EN$, $\ENP$, $\END$, $\ECN$, and $\ECNP$.
In Section \ref{fur}, we discuss concluding remarks.

\section{Preliminaries and background}\label{pre}
% Throughout this paper, let $T$ be a primitive recursively axiomatized consistent extension
% of Peano arithmetic $\PA$ in the language $\LA$ of first-order arithmetic.
% Let $\omega$ be the set of all natural numbers. For each $n \in \omega$,  $\num{n}$ denotes the numeral of $n$.
% In the present paper, we fix a natural G\"{o}del numbering such that if $\alpha$ is a proper sub-expression of a finite sequence $\beta$ of $\LA$-symbols, 
% then the G\"{o}del number of $\alpha$ is less than that of $\beta$. For each formula $\varphi$, let $\gn{\varphi}$ be the numeral of the G\"{o}del number of $\varphi$.
% Let $\{ \xi_t \}_{t \in \omega}$ denote the repetition-free primitive recursive enumeration of all $\LA$-formulas in ascending order of G\"{o}del numbers. 
% We note that if $\xi_u$ is a proper subformula of $\xi_v$, then $u<v$.

\subsection{Provability predicate and derivability conditions}
Throughout this paper, let $T$ be a primitive recursively axiomatized consistent extension
of Peano arithmetic $\PA$ in the language $\LA$ of first-order arithmetic.
Let $\omega$ be the set of all natural numbers. For each $n \in \omega$,  $\num{n}$ denotes the numeral of $n$.
For each formula $\varphi$, let $\gn{\varphi}$ be the numeral of the G\"{o}del number of $\varphi$.

We say that a formula $\PR_T(x)$ is a \textit{provability predicate} of $T$
if for any formula $\varphi$, $T \vdash \varphi $ if and only if $\PA \vdash \PR_T(\gn{\varphi})$.
Let $\Proof_T(x,y)$ denote a primitive recursive $\LA$-formula naturally expressing that ``$y$ is the G\"{o}del number of a $T$-proof of a formula whose G\"{o}del number is $x$.'' In this paper, we assume that $\Proof_T(x,y)$ is a single-conclusion, 
that is, $\PA \vdash \forall y \forall x_1 \forall x_2 (\Proof_T(x_1,y) \wedge \Proof_T(x_2,y) \to x_1=x_2)$ holds. holds.
Let $\Prov_T(x)$ denote the $\Sigma_1$ formula $\exists y \Proof_T(x,y)$. It is shown that the formula $\Prov_T(x)$ 
is a provability predicate.
Let $\Con_T$ denote the formula $\neg \Prov_T(\gn{0=1})$.
We introduce the following derivability conditions.
\begin{defn}[Derivability conditions]
Let $\PR_T(x)$ be a provability predicate of $T$.
\begin{itemize}
\item 
$\D{2}:$ $T \vdash \PR_T(\gn{\varphi \to \psi}) \to (\PR_T(\gn{\varphi}) \to \PR_T(\gn{\psi}))$
\item 
$\D{3}: T \vdash \PR_T(\gn{\varphi}) \to \PR_T(\gn{\PR_T(\gn{\varphi})})$.
\item 
$\mathbf{E}:$ If $T \vdash \varphi \leftrightarrow \psi$, then $T \vdash \PR_T(\gn{\varphi}) \leftrightarrow \PR_T(\gn{\psi})$. 
\item
$\mathbf{C}:$ $T \vdash \PR_T(\gn{\varphi}) \wedge \PR_T(\gn{\psi}) \to \PR_T(\gn{\varphi \wedge \psi})$.  
\end{itemize}
\end{defn}
We assume that the formula $\Prov_T(x)$ satisfies the condition $\D{2}$ and $\D{3}$.
Also, the conditions $\mathbf{E}$ and $\mathbf{C}$, which were introduced in \cite{Kur25}, are weaker than the condition $\D{2}$.

Let $\Con^L_T$, $\mathbf{Ros}$, and $\Con^S_T$ be the following consistency statements.
\begin{defn}
Let $\PR_T(x)$ be a provability predicate of $T$.
\begin{itemize}
    \item 
    $\Con^L_T : \equiv \neg \PR_T(\gn{0=1})$.
    \item
    $\mathbf{Ros}$: If $T \vdash \neg \varphi$, then $T \vdash \neg \PR_T (\gn{\varphi})$ for any formula $\varphi$.
    \item
    $\Con^S_T := \{ \neg(\PR_T(\gn{\varphi}) \wedge \PR_T(\gn{\neg \varphi})) \mid \varphi \ \text{is a formula} \}$.
\end{itemize}
\end{defn}
The first one $\Con^L_T$ is a consistency statement of the well-known second incompleteness theorem asserts that if $\PR_T(x)$ satisfies the conditions $\D{2}$ and $\D{3}$, then $T \nvdash  \Con^L_T$ holds.
The second one $\mathbf{Ros}$, which was introduced in \cite{Kur25}, corresponds to the property of Rosser provability predicate. The
modal counterpart  $(\mathsf{Ros}) \dfrac{\neg A}{\neg \Box A}$ of $\mathbf{Ros}$ was investigated in \cite{Kur23}.
The third statement $\Con^S_T$ called the schematic consistency statement was studied in \cite{Kur20-2}.

% $\Con^S_\PR$ was investigated in is called schematic consistency statement
% were investigated in \cite{Kur2,Kur25}. 

% \item 
% $\mathbf{M}:$ If $T \vdash \varphi \to \psi \Longrightarrow T \vdash \PR(\gn{\varphi}) \to \PR(\gn{\psi})$ for any formulas $\varphi$ and $\psi$. 
% The following conditions $\mathbf{E}$ and $\mathbf{C}$, which were introduced in \cite{Kur25}, are weaker than the condition $\D{2}$.
% \begin{defn}
% Let $\PR(x)$ be a provability predicate of $T$.
% \begin{itemize}
% \item 
% $\mathbf{E}:$ If $T \vdash \varphi \leftrightarrow \psi \Longrightarrow T \vdash \PR(\gn{\varphi}) \leftrightarrow \PR(\gn{\psi})$ for any formulas $\varphi$ and $\psi$. 
% \item 
% $\mathbf{C}:$ $T \vdash \PR(\gn{\varphi}) \wedge \PR(\gn{\psi}) \to \PR(\gn{\varphi \wedge \psi})$ for any formulas $\varphi$ and $\psi$.  
% \end{itemize}
% \end{defn}
The following shows the relationship between the three consistency statements. 
% The conditions $\mathbf{E}$ and $\mathbf{C}$ strengthen  consistency statements.
\begin{prop}\label{con}
For any provability predicate $\PR_T(x)$, the following hold:
\begin{itemize}
    \item 
If $T \vdash \Con^S_T$, then $\PR_T(x)$ satisfies $\mathbf{Ros}$. 
\item 
If $\PR_T(x)$ satisfies $\mathbf{C}$ and $\mathbf{Ros}$, then $T \vdash \Con^S_T$.
    \item 
If $\PR_T(x)$ satisfies $\mathbf{Ros}$, then $T \vdash \Con^L_T$

    \item 
If $\PR_T(x)$ satisfies $\mathbf{E}$ and $T \vdash \Con^L_T$, then $\PR_T(x)$  satisfies $\mathbf{Ros}$.
\end{itemize}
\end{prop}
With respect to $\mathbf{Ros}$ and $\Con^S_T$, the following versions of the second incompleteness theorem hold.
\begin{thm}\label{thm CD3}
Let $\PR_T(x)$ be any provability predicate of $T$.
\begin{itemize}
    \item
    If $\PR_T(x)$ satisfies $\mathbf{C}$ and $\D{3}$, then $\mathbf{Ros}$ does not hold for $\PR_T(x)$.
    \item 
     If $\PR_T(x)$ satisfies $\mathbf{E}$ and $\D{3}$, then $T \nvdash \Con^S_T$.
\end{itemize}

\end{thm}
By Proposition \ref{con} and Theorem \ref{thm CD3}, we obtain the following.
\begin{cor}\label{ECD3}
Let $\PR_T(x)$ be any provability predicate of $T$.
If $\PR_T(x)$ satisfies $\mathbf{E}$, $\mathbf{C}$, and $\D{3}$, then $T \nvdash \Con^L_T$.
\end{cor}
Kurahashi \cite{Kur25} proved that $\{ \mathbf{E}, \mathbf{C}, \mathbf{D3} \}$ is strictly weaker than $\{ \D{2}, \D{3}\}$. Thus, Corollary \ref{ECD3} is an improvement of the well-known second incompleteness theorem.

%The logic $\mathbf{K4}$ is obtained by adding the axiom scheme $\Box A \to \Box  \Box A$ to $\mathbf{K}$.
%The schemes $\Box (A \to B) \to (\Box A \to \Box B)$ and $\Box A \to \Box  \Box A$ are modal counter parts $\D{2}$ and $\D{3}$ respectively.

\subsection{Provability logic}
The language of modal propositional logic consists of  
propositional variables, the logical constant $\bot$, the logical connectives $\neg, \wedge, \vee, \to$,
and the modal operator $\Box$.
 Let $\MF$ denote the set of all modal propositional formulas.
We say that a modal logic $L$ is \textit{normal} if it contains all tautologies and the distribution axiom $\Box (A \to B) \to (\Box A \to \Box B)$ and is closed under Modus Ponens (\textrm{MP}) $\dfrac{A\to B \quad A}{B}$, Necessitation (\textrm{Nec}) $\dfrac{A}{\Box A}$, and uniform substitution. 
We say that a modal logic $L$ is \textit{non-normal} if the logic $L$ is not normal.
The weakest normal modal logic is called $\mathsf{K}$. 

% The axioms of the basic modal logic $\mathbf{K}$ consists of propositional tautologies in the language  $\mathcal{L}_\Box$
% and the distribution axiom scheme $\Box (A \to B) \to (\Box A \to \Box B)$. 
% The inference rules of modal logic $\mathbf{K}$ consists of Modus Ponens $(\mathrm{MP}) \ \dfrac{A \quad A \to B}{B}$
% and Necessitation $(\mathrm{Nec}) \ \dfrac{A}{\Box A}$.

For each provability predicate $\PR_T(x)$ of $T$,
we say that a mapping $f$ from $\MF$ to a set of $\LA$-sentences 
is an \textit{arithmetical interpretation} based on $\PR_T(x)$ if $f$ satisfies the following clauses:
\begin{itemize}
\item 
$f(\bot)$ is $0=1$, 
\item
$f(\neg A)$ is $\neg f(A)$,
\item 
$f(A \circ B)$ is $f(A) \circ f(B)$ for $\circ \in \{ \wedge , \vee , \to \}$, and
\item 
$f(\Box A)$ is $\PR_T(\gn{f(A)})$.
\end{itemize}
For any provability predicate $\PR_T(x)$ of $T$, 
let $\PL(\PR_T)$ denote the set of all modal formulas $A$ such that 
for any arithmetical interpretation $f$ based on $\PR_T(x)$, $T \vdash f(A)$.
We call the set $\PL(\PR_T)$ \textit{provability logic} of $\PR_T(x)$. 
A well-known result in the study of provability logics is Solovay's arithmetical completeness theorem.
The modal logic $\mathsf{GL}$ is obtained by adding the axiom scheme $\Box (\Box A \to A) \to \Box A$ to $\mathsf{K}$. 
\begin{thm}[Solovay \cite{Sol}]
If $T$ is $\Sigma_1$-sound, then $\PL(\Prov_T) = \mathsf{GL}$.
\end{thm}
Solovay proved the theorem by embedding  finite Kripke models of $\mathsf{GL}$ into arithmetic.
\subsection{Non-normal modal logics having neighborhood semantics}
 In this paper, we investigate the conditions $\mathbf{E}$ and $\mathbf{C}$ from the perspective of provability logic.
We introduce several non-normal modal logics that are closed under the rule $(\mathsf{RE}) \dfrac{A \leftrightarrow B}{\Box A \leftrightarrow \Box B}$.
\begin{defn}
The axiom and rules of the modal logic $\mathsf{EN}$ are as follows:
\begin{itemize}
    \item propositional tautologies.
    \item $(\mathsf{MP}) \ \dfrac{A \quad A \to B}{B}$.
    \item $(\mathsf{Nec}) \ \dfrac{A}{\Box A}$.
    \item $(\mathsf{RE}) \dfrac{A \leftrightarrow B}{\Box A \leftrightarrow \Box B}$.
\end{itemize}
Several extensions of $\EN$ are defined as follows:
\begin{itemize}
    \item $\ECN := \EN + (\Box A \wedge \Box B \to \Box (A \wedge B))$.
    \item $\ENP := \EN + \neg \Box \bot$.
    \item $\END := \EN + \neg (\Box A \land \Box \neg A)$.
    \item $\ECNP:= \ECN + \neg \Box \bot$.
\end{itemize}
\end{defn}
% The logic $\mathsf{ECN}$ is obtained from $\mathsf{EN}$ by adding the axiom $\Box A \wedge \Box B \to \Box (A \wedge B)$.
% The logic $\ENP$ and $\END$ are  obtained from $\mathsf{EN}$ by adding the axioms $\neg \Box \bot$ and $\neg (\Box A \wedge \Box \neg A)$, respectively. The logic $\mathsf{ECNP}$ is obtained from $\mathsf{ECN}$ by adding the axiom $\neg \Box \bot$.
Here, $\neg \Box \bot$ and $\neg (\Box A \wedge \Box \neg A)$ are modal counterparts of $\Con^L_T$ and $\Con^S_T$, respectively.
Adding the modal counterpart of $\mathbf{Ros}$ to $\mathsf{EN}$ leads to the same system as $\ENP$. Therefore, in what follows, we focus on modal counterparts of $\Con^L_T$ and $\Con^S_T$.
Moreover, note that over $\mathsf{ECN}$, the axioms $\neg \Box \bot$ and $\neg (\Box A \land \Box \neg A)$ become equivalent.
% \begin{defn}
% \leavevmode
% \begin{itemize}
%     \item $\EN = \N + \dfrac{A \leftrightarrow B}{\Box A \leftrightarrow \Box B}$.
%     \item $\ECN = \EN+ \Box A \wedge \Box B \to \Box (A \wedge B)$. 
%     \item $\ENP = \EN + \neg \Box \bot$.
% \item $\END = \EN + \neg (\Box A \wedge \Box \neg A)$.
% \item $\ECNP = \ECN + \neg \Box \bot$.    
% \end{itemize}
% \end{defn}

The logic $\EN$ does not contain the distribution axiom, and does not admit Kripke semantics.
We introduce neighborhood semantics, which 
provides a suitable framework for non-normal modal logics, including systems such as $\EN$, $\ECN$, and others.
For references on neighborhood semantics, see \cite{Che,Eric}.

\begin{defn}[$\mathsf{EN}$-frame]
A pair $(W,N)$ is called an $\mathsf{EN}$-frame if 
the following conditions hold:
\begin{itemize}
    \item 
$W$ is a non-empty set.
    \item 
$N$ is a function $N : W \to \mathcal{P (\mathcal{P}}(W))$ such that for any $x \in W,$ $W \in N(x)$.
\end{itemize}

\end{defn}

\begin{defn}[$\EN$-model]
A triple $(W,N, v)$ is an $\EN$-model if $(W,N)$ is an $\EN$-frame and $v$ is a function from $\MF$ to $\mathcal{P}(W)$ satisfying the following conditions: for any $x \in W$,
\begin{itemize}
    \item $x \notin v(\bot)$. 
    \item $x \in v(A \wedge B) \iff x \in v(A) \cap v(B)$.
    \item 
    $x \in v(\neg A) \iff x \notin v(A)$.
    \item 
    $x \in v(A \vee B) \iff x \in v(A) \cup v(B)$.
    \item 
    $x \in v(A \to B) \iff$ if $x \in v(A)$, then $x \in v(B)$.
    \item
    $x \in v(\Box A) \iff v(A) \in N(x)$.
\end{itemize}

\end{defn}

\begin{defn}
\leavevmode

\begin{itemize}
\item 
A formula $A$ is valid in an $\EN$-model $(W,N,v)$ if for any $x \in W$, $x \in v(A)$.
\item 
A formula $A$ is valid in an $\EN$-frame $(W,N)$ if $A$ is valid in any $\EN$-model  whose frame is $(W,N)$.
\end{itemize}
\end{defn}
We introduce some frame properties.
\begin{defn}[$\ECN$, $\ENP$, $\END$, and $\mathsf{ECNP}$-frames]
\leavevmode

\begin{itemize}
    \item 
An $\EN$-frame $(W,N)$ is an $\ECN$-frame if for any $x \in W$ and $U, V \in \mathcal{P}(W)$, if $U, V \in N(x)$, then $U \cap V \in N(x)$. 
    \item 
An $\EN$-frame $(W,N)$ is an $\ENP$-frame if for any $x \in W$, $\emptyset \notin N(x)$.
\item 
An $\EN$-frame $(W,N)$ is an $\END$-frame if for any $x \in W$ and $V \in \mathcal{P}(W)$, if $V \in N(x)$, then $W \setminus V \notin N(x)$.
\item 
An $\EN$-frame $(W,N)$ is an $\ECNP$-frame if $(W,N)$ is $\ENP$-frame and $\ECN$-frame.
\end{itemize}
\end{defn}

As in the case of Kripke semantics, the validity of an axiom is characterized by a property of frames.
\begin{prop}[Cf.~\cite{Che}]
Let $(W,N)$ be an $\EN$-frame.
\begin{itemize}
    \item 
 $\Box p \wedge \Box q \to \Box (p \wedge q)$ is valid in $(W,N)$ if and only if $(W,N)$ is an $\ECN$-frame.   
    \item 
$\neg \Box \bot$ is valid in $(W,N)$ if and only if $(W,N)$ is an $\ENP$-frame.
    \item 
$\neg (\Box p \wedge \Box \neg p)$ is valid in $(W,N)$ 
if and only if $(W,N)$ is an $\END$-frame.
\end{itemize}

\end{prop}
A modal formula is \text{non-iterative} (see \cite{Eric}) if no modal operator occurs within the scope of another.
For example, the formulas
$\Box p \wedge \Box q \to \Box(p \wedge q)$,  $\neg \Box \bot$, and $\neg (\Box p \wedge \Box \neg p)$ are non-iterative.
A logic $L$ is called non-iterative if $L$ extends $\mathsf{EN}$ and $L$ is axiomatized by non-iterative axioms.
Thus, the systems $\EN$, $\ENP$, $\END$, $\ECN$, and $\ECNP$ are all non-iterative.
For each non-iterative logic $L$ which is finitely axiomatizable, 
Lewis \cite{lew} proved that $L$ has finite frame property with respect to neighborhood semantics. In particular, we obtain the following.
\begin{thm}[Cf. {\cite[Theorem 2.50]{Eric}}]\label{lew}
Let $L \in \{ \EN, \ENP, \END, \ECN, \ECNP   \}$ and $A \in \MF$. The following are equivalent:
\begin{enumerate}
    \item $L \vdash A$.
    \item $A$ is valid in any $L$-frame.
    \item $A$ is valid in any finite $L$-frame.
\end{enumerate}
\end{thm}
From the proof of Theorem \cite{lew}, the sets of all theorems of $\EN$, $\ENP$, $\END$, $\ECN$, and $\ECNP$ are primitive recursive.
\begin{cor}\label{Lew2}
For each $L \in \{ \EN, \ENP, \END, \ECN, \ECNP   \}$, there exists a primitive recursive decision procedure for provability in $L$. 
\end{cor}

%\section{Summary of results}
%\subsection{Provability logic}

\section{Arithmetical completeness of $\mathsf{EN}$, $\mathsf{ENP}$, and $\mathsf{END}$}\label{pr1}
In this section, we prove the arithmetical completeness theorems for $\mathsf{EN}$, $\mathsf{ENP}$, and $\mathsf{END}$. Before proving the theorems, we introduce some notions used throughout this paper.

We call an $\LA$-formula  \textit{propositionally atomic} if it is either atomic or of the form $Q x \psi$, where $Q \in \{\forall, \exists  \}$.
For each propositionally atomic formula $\varphi$, we prepare a propositional variable $p_{\varphi}$.
Let $I$ be a primitive recursive injection from $\LA$-formulas into propositional formulas, which is defined as follows:
\begin{itemize}
\item 
$I(\varphi)$ is $p_{\varphi}$ for each propositionally atomic formula $\varphi$,
\item 
$I(\neg \varphi)$ is $\neg I(\varphi)$,
\item 
$I(\varphi \circ \psi)$ is $I(\varphi) \circ I(\psi)$ for $\circ \in \{ \wedge, \vee, \to \}$.
\end{itemize}
Let $X$ be a finite set of $\LA$-formulas.
An $\LA$-formula $\varphi$ is called a \textit{tautological consequence} of $X$
if $\bigwedge_{\psi \in X}I(\psi) \to I(\varphi)$ is a tautology.  
Let $X \tc \varphi$ denote that $\varphi$ is a tautological consequence of $X$.
For each $n \in \omega$, let $F_n$ be the set of all $\LA$-formulas whose G\"{o}del numbers are less than or equal to $n$. Let $\mathbb{N}$ be the standard model of arithmetic and we define $P_{T,n} : = \{ \varphi \in F_{n} \mid \mathbb{N} \models \exists y \leq \num{n} \ \Proof_T(\gn{\varphi}, y)  \}$. 
We see that  if 
$P_{T,n} \tc \varphi$, then $\varphi$ is provable in $T$.
The above notions can be formalized in $\PA$.

For each $m \in \omega$, we define a binary relation $\leftrightarrow_{m}$ on $\LA$-formulas as follows.
\begin{align*}
\varphi \leftrightarrow_m \psi \ \text{if and only if there exists  a finite sequence} \ \psi_0, \ldots, \psi_k \ \\
\text{such that} \ \varphi \equiv \psi_0,  \psi \equiv \psi_k, \  \text{and} \ \text{for each} \ i<k, \psi_i \leftrightarrow \psi_{i+1} \in P_{T,m}.
\end{align*}
The primitive recursiveness of the relation $\{(\varphi, \psi,m)\mid \varphi\leftrightarrow_m \psi  \}$ is easily proved.
We obtain the following properties.

\begin{prop}
Let $m \in \omega$ and $\varphi$ and $\psi$ be formulas.
\begin{enumerate}
\item 
If $P_{T,m} \tc \varphi$, then $\varphi$ is provable in $T$. 
\item 
If $\varphi \leftrightarrow_{m} \psi$, then $P_{T,m} \tc \varphi \leftrightarrow \psi$.
\item 
The binary relation $\leftrightarrow_{m}$ is transitive.

\end{enumerate}
\end{prop}

Next, we prepare a primitive recursive function $h$, which is originally introduced in \cite{Kur20}.
The function $h$ is defined by using the recursion theorem as follows:
\begin{itemize}
	\item $h(0) = 0$. 
	\item $h(s+1) = \begin{cases} \text{min} \ J_s & \text{if}\ h(s) = 0 \ \text{and} \ J_s \neq \emptyset, \\
				
			h(s) & \text{otherwise},
		\end{cases}$
\end{itemize}
where $J_s = \{ j \in \omega \setminus \{0\} \mid P_{T,s} \tc \neg \lambda(\num{j}) \}$ and $\lambda(x)$ is the $\Sigma_1$ formula $\exists y(h(y) = x)$. 
The function $h$ satisfies the property that for each $s$ and $i$, $h(s+1)=i$ implies $i \leq s+1$, which guarantees that the function $h$ is primitive recursive (See \cite{Kur20}, p. 603).
The following proposition holds for $h$. 
\begin{prop}[Cf.~{\cite[Lemma 3.2.]{Kur20}}]\label{Prop:h}
\leavevmode
\begin{enumerate}
	\item $\PA \vdash \forall x \forall y(0 < x < y \land \lambda(x) \to \neg \lambda(y))$. 
	\item $\PA \vdash \neg \Con_T \leftrightarrow \exists x(\lambda(x) \land x \neq 0)$. 
	\item For each $i \in \omega \setminus \{0\}$, $T \nvdash \neg \lambda(\num{i})$. 
	\item For each $l \in \omega$, $\PA \vdash \forall x \forall y(h(x) = 0 \land h(x+1) = y \land y \neq 0 \to x \geq \num{l})$. 
\end{enumerate}
\end{prop}

We are ready to prove the following uniform version of the arithmetical completeness theorems for $\mathsf{EN}$, $\mathsf{ENP}$, and $\mathsf{END}$. 
\begin{thm}\label{thmEN}
For $L \in \{ \mathsf{EN}, \mathsf{ENP}, \mathsf{END}    \}$, 
there exists a $\Sigma_1$ provability predicate $\PR_T(x)$ of $T$ satisfying $\mathbf{E}$ such that
\begin{enumerate}
    \item 
    for any $A \in \MF$ and any arithmetical interpretation $f$ based on $\PR_T(x)$, if $L \vdash A$, then $T \vdash f(A)$, and
    \item 
    there exists an arithmetical interpretation $f$ based on $\PR_T(x)$ such that for any $A \in \MF$, $L \vdash A$ if and only if $T \vdash f(A)$.
\end{enumerate}

\end{thm}
\begin{proof}
Let $L \in \{ \mathsf{EN}, \mathsf{ENP}, \mathsf{END}  \}$.
By Fact \ref{Lew2}, we obtain a primitive recursive enumeration $\langle A_k \rangle_{k \in \omega}$ of all $L$-unprovable formulas.
For each $A_k$, we can primitive recursively construct a finite $L$-model $(W_k, N_k,v_k)$ falsifying $A_k$.
We may assume that the sets $\{W_k\}_{k \in \omega}$ are pairwise disjoint and $\bigcup_{k \in \omega} W_k= \omega \setminus \{0\}$.
We may also assume that 
 $\langle (W_k, N_k, v_k) \rangle_{k \in \omega}$ is primitive recursively represented in $\PA$ and some basic properties of this enumeration are provable in $\PA$.
 
 % enumeration $\langle (W_k, N_k, v_k) \rangle_{k \in \omega}$ of pairwise disjoint finite $L$-models such that $\bigcup_{n \in \omega} W_n= \omega \setminus \{0\}$, and if $L \nvdash A$, then there exist $n \in \omega$ and $i \in W_n$ such that $i \notin v_n(A)$.
We define the primitive recursive function $g_0$ outputting all $T$-provable formulas step by step.
The definition of $g_0$ consists of Procedures 1 and 2, and starts with Procedure 1.
In Procedure 1, we define the values $g_0(0), g_0(1), \ldots$ 
 by referring to the values $h(0), h(1), \ldots$ and $T$-proofs based on $\Proof_T(x,y)$. 
At the first time $h(s+1) \neq 0$, the definition of $g_0$ switches to Procedure 2 at Stage $s$. The function $g_0$ is defined as follows.

% $f_0(p) \equiv \exists x \exists y( x \in W_y \wedge \lambda(x) \wedge x \neq 0 \wedge p \in v_y(p))$.
\medskip

\textbf{Procedure 1.}

Stage $s$:
\begin{itemize}
\item If $h(s+1) =0$,
\begin{equation*}
  g_0(s)  = \begin{cases}
       \varphi & \text{if}\ s\ \text{is a}\ T \text{-proof of}\ \varphi \\
               0 & \text{otherwise}.
             \end{cases}
  \end{equation*}

Then, go to Stage $s+1$.

\item If $h(s+1) \neq 0$, go to Procedure 2.
\end{itemize}

\textbf{Procedure 2.}

Suppose $s$ and $i \neq 0$ satisfy $h(s)=0$ and $h(s+1)=i$. 
Let $k$ be the number such that $i \in W_k$. 
Let $\chi_{0}, \ldots , \chi_{l-1}$ be the enumeration of all formulas in $X \cup Y$, where
\[
X = \{ \psi \in F_{s-1} \mid \exists \varphi \in P_{T,s-1} \ \text{such that} \ \varphi \leftrightarrow_{s-1} \psi \}
\]
and 
\begin{align*}
Y = \{ \psi \in F_{s-1} \mid  \exists \varphi & \in P_{T,s-1} \ \exists V \in N_k(i) [ \forall j \in V (\lambda(\num{j}) \to \varphi \in P_{T,s-1}) \\
& \& \  \forall j \in W_k \setminus V (\lambda(\num{j}) \to \neg \varphi \in P_{T,s-1})  \ \& \ \varphi \leftrightarrow_{s-1} \psi \ ]\}.
\end{align*}
\begin{itemize}
    \item 
For each $t< l$, let $g_0(s+t) = \chi_t$.
    \item
For each $t \geq l$, let $g_0(s+t) = 0$.   
\end{itemize}
The construction of $g_0$ has just been  finished. We define $\PR_{g_0}(x) \equiv \exists y (g_0(y)=x)$.

% The following claim ensures that the formula $\PR_{g_0}(x)$ is a $\Sigma_1$ provability predicate of $T$.
\begin{cl}\label{EN1}
Let $\varphi$ be an $\LA$-formula. Then, $\PA + \Con_T \vdash \PR_{g_0}(\gn{\varphi}) \leftrightarrow \Prov_T (\gn{\varphi})$.
\end{cl}
\begin{proof}
We work in $\PA + \Con_{T}$: By Proposition \ref{Prop:h}.2, the definition of $g_0$ never switches to Procedure 2.
Hence, for any $s$, $g_0(s)= \varphi$ if and only if $s$ is a proof of $\varphi$ in $T$. 
\end{proof}
Therefore, $\PR_{g_0}(x)$  is a $\Sigma_1$ provability predicate of $T$.
We prove $\PR_{g_0}(x)$ satisfies the condition $\mathbf{E}$.
\begin{cl}\label{EN2}
Let $\varphi$ and $\psi$ be $\LA$-formulas. If $T \vdash \varphi \leftrightarrow \psi$, then $\PA \vdash \PR_{g_0}(\gn{\varphi}) \leftrightarrow \PR_{g_0}(\gn{\psi})$.
    
\end{cl}
\begin{proof}
Suppose $T \vdash \varphi \leftrightarrow \psi$. 
%Let $p \in \omega$ be a $T$-proof of $\varphi \leftrightarrow \psi$.
By Claim \ref{EN1}, we obtain $\PA + \Con_T \vdash \PR_{g_0}(\gn{\rho}) \leftrightarrow \Prov_T(\gn{\rho})$ for any formula $\rho$. Since $\Prov_T(x)$ satisfies the condition $\mathbf{E}$, we obtain $\PA \vdash \Prov_T(\gn{\varphi}) \leftrightarrow \Prov_T(\gn{\psi})$. Thus, it follows that $\PA + \Con_T \vdash \PR_{g_0}(\gn{\varphi}) \leftrightarrow \PR_{g_0}(\gn{\psi})$.

Next, we prove $\PA + \neg \Con_T \vdash \PR_{g_0}(\gn{\varphi}) \leftrightarrow \PR_{g_0}(\gn{\psi})$. By Proposition \ref{Prop:h}.2, it suffices to prove that $\PA + \exists x (\lambda(x) \wedge x \neq 0) \vdash \PR_{g_0}(\gn{\varphi}) \leftrightarrow \PR_{g_0}(\gn{\psi})$.
We only prove that $\PA + \exists x (\lambda(x) \wedge x \neq 0) \vdash \PR_{g_0}(\gn{\varphi}) \to \PR_{g_0}(\gn{\psi})$.
Let $p \in \omega$ be a $T$-proof of $\varphi \leftrightarrow \psi$.

We work in $\PA + \exists x (\lambda(x) \wedge x \neq 0)$: Let $i \neq 0$, $k$ and $s$ be such that $h(s) = 0$ and $h(s+1) = i \in W_k$.
By Proposition \ref{Prop:h}.4, the number $s$ is non-standard and we obtain $p\leq s-1$.
Thus, it follows that $\varphi \leftrightarrow_{s-1} \psi$. 
Suppose $\PR_{g_0}(\gn{\varphi})$, that is, $\varphi$ is output by $g_0$.
If $\varphi$ is output in Procedure 1, then we obtain $\varphi \in P_{T,s-1}$.
Since $\varphi \leftrightarrow_{s-1} \psi$, it follows that $\psi \in X$. Thus, $\psi$ is output by $g_0$ in Procedure 2.
If $\varphi$ is output in Procedure 2, then we obtain $\varphi \in X \cup Y$. 
If $\varphi \in X$, then there exists a formula $\rho \in P_{T,s-1}$ such that $\rho \leftrightarrow_{s-1} \varphi$.
By $\varphi \leftrightarrow_{s-1} \psi$ and the transitivity of $\leftrightarrow_{s-1}$, we obtain $\rho \leftrightarrow_{s-1} \psi$.
Thus, $\psi \in X$ holds.
If $\varphi \in Y$, then we obtain $\psi \in Y$ by the similar argument as in the case $\varphi \in X$.
In either case, we obtain $\psi \in  X \cup Y$. Therefore, $\psi$ is output by $g_0$ in Procedure 2, that is, we obtain $\PR_{g_0}(\gn{\psi})$.

Thus, by the law of excluded middle, we obtain $\PA \vdash \PR_{g_0}(\gn{\varphi}) \leftrightarrow \PR_{g_0}(\gn{\psi})$.
\end{proof}
If $L$ is either $\mathsf{ENP}$ or $\mathsf{END}$,
 consistency statements $\Con^L_{T}$ and
 $\Con^S_{T}$
are provable in $T$, respectively.
\begin{cl}\label{ENP}
If the logic $L$ is $\mathsf{ENP}$, then $\PA \vdash \neg \PR_{g_0}(\gn{0=1})$.
\end{cl}

\begin{proof}
By Claim \ref{EN1}, we obtain $\PA + \Con_T \vdash \neg \PR_{g_0}(\gn{0=1})$, and therefore it suffices to prove that $\PA + \exists x (\lambda(x) \wedge x \neq 0) \vdash \neg \PR_{g_0}(\gn{0=1})$.

We work in $\PA + \exists x (\lambda(x) \wedge x \neq 0)$:
Let $s$, $i \neq 0$, and $k$ be such that $h(s)=0$ and $h(s+1)=i \in W_k$. Suppose, towards a contradiction, that $\PR_{g_0}(\gn{0=1})$ holds, that is, $0=1$ is output by $g_0$. If $0=1$ is in $P_{T, s-1} \cup X$, then $P_{T,s-1} \tc 0=1$ because every element of $X$ is a tautological consequence of $P_{T,s-1}$.
Thus, we obtain $P_{T,s-1} \tc \neg \lambda(\num{i})$. This contradicts  $h(s)=0$. Therefore we obtain $0=1 \in Y$.
Then, there exists a formula $\varphi$ and a set $V \in N_k(i)$ such that 
$\varphi \leftrightarrow_{s-1} 0=1$ and
$\lambda(\num{j}) \to \varphi \in P_{T,s-1}$ for any $j \in V$.
By the frame condition of $\mathsf{ENP}$, the set $V \in N_{k}(\num{i})$ is non-empty.
Then, there exists a $j \in V$ such that $\lambda(\num{j}) \to \varphi \in P_{T, s-1}$. Since $\varphi \leftrightarrow_{s-1} 0=1$, we have $P_{T,s-1} \tc \neg \lambda(\num{j})$. This contradicts $h(s)=0$. Therefore, we conclude that $\neg \PR_{g_0}(\gn{0=1})$ holds.

By the law of excluded middle, we obtain $\PA \vdash \neg \PR_{g_0}(\gn{0=1})$.
\end{proof}

\begin{cl}\label{END}
If $L$ is $\mathsf{END}$, then for any $\LA$-formula $\varphi$, $\PA \vdash \neg \bigl( \PR_{g_0}(\gn{\varphi}) \wedge \PR_{g_0}(\gn{\neg \varphi}) \bigr)$ holds.
\end{cl}
\begin{proof}
Let $\varphi$ be any formula.
Since $\PA + \Con_T \vdash \neg \bigl( \Prov_{T}(\gn{\varphi}) \wedge \Prov_T(\gn{\neg \varphi}) \bigr)$, by Claim \ref{EN1}, we obtain $\PA + \Con_T \vdash \neg \bigl( \Pr_{g_0}(\gn{\varphi}) \wedge \Pr_{g_0}(\gn{\neg \varphi}) \bigr)$. It suffices to prove that $\PA + \exists x (\lambda(x) \wedge x \neq 0) \vdash \neg \bigl( \Pr_{g_0}(\gn{\varphi}) \wedge \Pr_{g_0}(\gn{\neg \varphi}) \bigr)$.

We work in $\PA + \exists x (\lambda(x) \wedge x \neq 0)$: Let $s$, $i \neq 0$, and $k$ be such that $h(s)=0$ and $h(s+1)=i \in W_k$.
Suppose, towards a contradiction, that $\PR_{g_0}(\gn{\varphi}) \wedge \PR_{g_0}(\gn{\neg \varphi})$ holds. Then, we obtain $\varphi, \neg \varphi \in P_{T, s-1} \cup X \cup Y$.
We distinguish the following four cases.
\begin{enumerate}
\item[(i)] $\varphi, \neg \varphi \in P_{T,s-1} \cup X$: Then, $P_{T,s-1}$ is inconsistent. Thus, $\neg \lambda(\num{i})$ is a t.c.~of $P_{T,s-1}$, and this contradicts $h(s)=0$.

\item[(ii)] $\varphi \in P_{T,s-1} \cup X$ and $\neg \varphi \in Y$:
Then, there exists a set $V \in N_k(i)$ such that
$P_{T,s-1} \tc \lambda(\num{j}) \to \neg \varphi$
for any 
$j \in V$.
Since $(W_k, N_k)$ is an $\EN$-frame, we obtain
$W \in N_k(i)$, and the frame condition of $\mathsf{END}$ implies $\emptyset \notin N_k(i)$. Thus, $V$ is non-empty and there exists a number $j \in V$ such that $P_{T,s-1} \tc \lambda(\num{j}) \to \neg \varphi$. Thus, it follows that $P_{T,s-1} \tc \varphi \to  \neg \lambda(\num{j})$.
By $\varphi \in P_{T,s-1} \cup X$, we obtain $P_{T,s-1} \tc \neg \lambda(\num{j})$. This contradicts $h(s)=0$. 

\item[(iii)] $\neg \varphi \in P_{T,s-1} \cup X$ and $\varphi \in Y$:
By the similar argument as in the case (ii), we obtain a contradiction.
\item[(iv)] $\varphi, \neg \varphi \in Y$:
Then, there exists $V_0, V_1 \in \mathcal{P}(W_k)$ such that $V_0, V_1 \in N_k(i)$, 
and the following conditions hold:
\begin{itemize}
\item 
$P_{T,s-1} \tc \lambda(\num{j}) \to \varphi$ for any $j \in V_0$ and \\$P_{T,s-1} \tc \lambda(\num{j}) \to \neg \varphi$ for any $j \in W_k \setminus V_0$.
\item 
$P_{T,s-1} \tc \lambda(\num{j}) \to \neg \varphi$ for any $j \in V_1$ and \\ $P_{T,s-1} \tc \lambda(\num{j}) \to \neg \neg \varphi$ for any $j \in W_k \setminus V_1$.
\end{itemize}
We prove $V_0 = W_k \setminus V_1$. Suppose, towards a contradiction, that $V_0 \neq W_k \setminus V_1$.
Then it follows that $V_0 \not \subseteq W_k \setminus V_1$ or $W_k \setminus V_1 \not \subseteq V_0$ hold.
If $V_0 \not \subseteq W_k \setminus V_1$,  then, there exists a number $j \in V_0 \cap V_1$ such that $P_{T, s-1} \tc (\lambda(\num{j}) \to \varphi) \wedge (\lambda(\num{j}) \to \neg \varphi)$.
Therefore, we obtain $P_{T,s-1} \tc \neg \lambda(\num{j})$. This contradicts $h(s)=0$.
If $W_k \setminus V_1 \not \subseteq V_0$, then by the similar argument as in the case $V_0 \not \subseteq W_k \setminus V_1$, we obtain a contradiction.
Thus, it follows that $V_0 = W_k \setminus V_1$.
Since $V_0, V_1 \in N_k(i)$, we obtain $(W_k \setminus V_1), V_1 \in N_k(i)$. This contradicts the frame condition of $\mathsf{END}$.
\end{enumerate}
Thus, we conclude that $\neg(\PR_{g_0}(\gn{\varphi}) \wedge \PR_{g_0}(\gn{\neg \varphi}))$ holds.

By the law of excluded middle, we obtain $\PA \vdash \neg(\PR_{g_0}(\gn{\varphi}) \wedge  \PR_{g_0}(\gn{\neg \varphi}))$.
\end{proof}
For each propositional variable $p$, we define an arithmetical interpretation $f_0$ by $f_0(p) \equiv \exists x \exists y( x \in W_y \wedge \lambda(x) \wedge x \neq 0 \wedge x \in v_y(p))$.

\begin{cl}\label{cl 5.15}\label{EN3}
Let $B \in \MF$ and $i \in W_k$.
\begin{enumerate}
\item
If $i \in v_k(B)$, then $\PA \vdash \lambda(\num{i}) \to f_0(B)$.
\item 
If $i \notin v_k(B)$, then $\PA \vdash \lambda(\num{i}) \to \neg f_0(B)$. 
\end{enumerate}
\end{cl}
\begin{proof}
We prove Clauses 1 and 2 simultaneously by induction on the construction of $B$. 
We prove only the case $B \equiv \Box C$.

1. Suppose $i \in v_k(\Box C)$.
Then, we obtain $v_k(C) \in N_k(i)$, and by the induction hypothesis, it follows that $\PA \vdash \lambda(\num{j}) \to f_0(C)$ for any $j \in v_k(C)$ and $\PA \vdash \lambda(\num{j}) \to \neg f_0(C)$ for any $j \in W_k \setminus v_k(C)$.
Since the set $W_k$ is finite,
there exists $p \in \omega$ such that 
% for any $j \in v_k(C)$, $\lambda(\num{j}) \to f_0(C) \in P_{T,p}$ and for any $j \in W_k \setminus v_k(C)$, $\lambda(\num{j}) \to \neg f_0(C) \in P_{T,p}$.
$\PA$ proves the following:
\begin{align}\label{5.1}
\ v_k(C) \in N_k(i) & \wedge \forall y \in v_k(C)  (\lambda(y) \to f_0(C)  \in P_{T,p}) \notag \\  
& \wedge \forall y \in W_k \setminus v_k(C)(\lambda(y) \to \neg f_0(C) \in P_{T,p}).
\end{align}

We work in $\PA+\lambda(\num{i})$:
Let $s$ be such that $h(s)=0$ and $h(s+1) = i \neq 0$.
Since   (\ref{5.1}), $f_0(C) \leftrightarrow f_0(C) \in P_{T,s-1}$, and $s-1 \geq p$ hold,  we obtain $f_0(C) \in Y$.
Therefore, $f_0(C)$ is output by $g_0$ in Procedure 2, that is, $\PR_{g_0}(\gn{f_0(C)})$ holds.

2. Suppose $i \notin v_k(\Box C)$. 
Then, $v_k(C) \notin N_k(i)$ holds, and for each $V \in N_k(i)$, we have either $V \not\subseteq v_k(C)$ or $v_k(C) \not\subseteq V$.
Hence, there exists some $j$ such that 
$j \in V \setminus v_k(C)$ or
$j \in v_k(C) \setminus V$.
By the induction hypothesis, this implies that 
\begin{itemize}
  \item there exists $j \in V$ such that $\PA \vdash \lambda(\num{j}) \to \neg f_0(C)$ or
  \item there exists $j \in W_k \setminus V$ such that $\PA \vdash \lambda(\num{j}) \to f_0(C)$.
\end{itemize}
Since the set $N_k(i)$ is finite, there exists a number $p \in \omega$ such that $\PA$ proves the following:
\begin{align}\label{5.2}
\forall V  \in N_k(i) [& \exists y \in V(\lambda(y) \to \neg f_0(C) \in P_{T,p}) \notag \\ 
& \vee \exists y \in W_k \setminus V(\lambda(y) \to  f_0(C) \in P_{T,p})]. 
\end{align}
We work in $\PA + \lambda(\num{i})$:
Let $s$ be such that $h(s)=0$ and $h(s+1)=i \in W_k$.
We prove that $f_0(C)$ is not output by $g_0$, that is, $f_0(C) \notin P_{T,s-1} \cup X \cup Y$.
Suppose, towards a contradiction, that $f_0(C) \in P_{T,s-1} \cup X \cup Y$. We consider the following two cases.
\begin{enumerate}
    \item[(i)] $f_0(C) \in P_{T,s-1} \cup X$:
    Then, $P_{T,s-1} \tc f_0(C)$ holds.
Since $W_k \in N_k(i)$,
 there exists $j \in W_k$ such that $\lambda(\num{j}) \to \neg f_0(C) \in P_{T,s-1}$ by (\ref{5.2}). Therefore, $\neg \lambda(\num{j})$ is a t.c.~of $P_{T,s-1}$. This contradicts $h(s)=0$.
\item[(ii)] 
$f_0(C) \in Y$:
Then, there exists $V \in N_k(i)$ such that $P_{T,s-1} \tc \lambda(\num{j}) \to f_0(C)$ for any $j \in V$, and 
 $P_{T,s-1} \tc \lambda(\num{j}) \to \neg f_0(C)$ for any $j \in W_k \setminus V$.
By (\ref{5.2}), there exists $j \in V$ such that $P_{T,s-1} \tc \lambda(\num{j}) \to \neg f_0(C)$. Since $P_{T,s-1} \tc \lambda(\num{j}) \to f_0(C)$ holds, we obtain $P_{T,s-1} \tc \neg \lambda(\num{j})$. This contradicts $h(s)=0$.
\end{enumerate}
Thus, we obtain $f_0(C) \notin P_{T,s-1} \cup X \cup Y$, and $f_0(C)$ is not output by $g_0$.
\end{proof}
We complete our proof of Theorem \ref{thmEN}. 
By Claims \ref{EN2}, \ref{ENP}, and  \ref{END} the implication $(\Rightarrow)$ of Theorems \ref{thmEN}.1 and \ref{thmEN}.2 is obvious.
We prove the implication $(\Leftarrow)$ of Theorems \ref{thmEN}.1 and \ref{thmEN}.2. Suppose $L \nvdash A$. Then, $A \equiv A_k$ for some $k \in \omega$ and $i \notin v_k(A)$ for some $i \in W_k$. Hence we obtain $\PA \vdash \lambda(\num{i}) \to \neg f_0(A)$ by Claim \ref{EN3}. Thus, we obtain $T \nvdash f_0(A)$ by Proposition \ref{Prop:h}.3.
\end{proof}
 % \item $\mathsf{EN} \vdash A$.
%     \item
%     $A \in \PL(\PR)$ for any provability predicate $\PR(x)$ of $T$ satisfying $\mathbf{E}$. 
%     \item 
%     $A \in \PL(\PR)$ for any $\Sigma_1$ provability predicate $\PR(x)$ of $T$ satisfying $\mathbf{E}$. 
% \end{itemize}
\begin{cor}[The arithmetical completeness of $\mathsf{EN}$]
\begin{align*}
    \EN & = \bigcap \{ \PL(\PR_T) \mid \PR_T(x) \text{ is a provability predicate satisfying } \mathbf{E} \}\\
    & =  \bigcap \{ \PL(\PR_T) \mid \PR_T(x) \text{ is a } \Sigma_1 \text{ provability predicate satisfying } \mathbf{E} \}.
\end{align*}
Moreover, there exists a $\Sigma_1$ provability predicate $\PR_T(x)$ of $T$ such that $\mathsf{EN} = \PL(\PR_T)$.
\end{cor}

\begin{cor}[The arithmetical completeness of $\mathsf{ENP}$]
\begin{align*}
    \ENP & = \bigcap \{ \PL(\PR_T) \mid \PR_T(x) \text{ satisfies } \mathbf{E} \text{ and } T \vdash \Con^L_T \},\\
    & = \bigcap \{ \PL(\PR_T) \mid \PR_T(x) \text{ is } \Sigma_1 \text{ and satisfies } \mathbf{E} \text{ and } T \vdash \Con^L_T \}.
\end{align*}
Moreover, there exists a $\Sigma_1$ provability predicate $\PR_T(x)$ of $T$ such that $\mathsf{ENP} = \PL(\PR_T)$.
\end{cor}

\begin{cor}[The arithmetical completeness of $\mathsf{END}$]
\begin{align*}
    \END & = \bigcap \{ \PL(\PR_T) \mid \PR_T(x) \text{ satisfies } \mathbf{E} \text{ and } T \vdash \Con^S_T \},\\
    & = \bigcap \{ \PL(\PR_T) \mid \PR_T(x) \text{ is } \Sigma_1 \text{ and satisfies } \mathbf{E} \text{ and } T \vdash \Con^S_T \}.
\end{align*}
Moreover, there exists a $\Sigma_1$ provability predicate $\PR_T(x)$ of $T$ such that $\mathsf{END} = \PL(\PR_T)$.
\end{cor}

% \begin{cor}[Arithmetical completeness of $\mathsf{ENP}$ ]
% For any $A \in \MF$, the following are equivalent:
% \begin{itemize}
%     \item $\mathsf{ENP} \vdash A$.
%     \item
%     $A \in \PL(\PR)$ for any provability predicate $\PR(x)$ of $T$ satisfying $\mathbf{E}$ and $T \vdash \neg \PR(\gn{0=1})$. 
%     \item 
%     $A \in \PL(\PR)$ for any $\Sigma_1$ provability predicate $\PR(x)$ of $T$ satisfying $\mathbf{E}$ and $T \vdash \neg \PR(\gn{0=1})$. 
% \end{itemize}
% Moreover, there exists a $\Sigma_1$ provability predicate $\PR(x)$ of $T$ such that $\mathsf{ENP} = \PL(\PR)$.
% \end{cor}

% \begin{cor}[Arithmetical completeness of $\mathsf{END}$ ]
% For any $A \in \MF$, the following are equivalent:
% \begin{itemize}
%     \item $\mathsf{END} \vdash A$.
%     \item
%     $A \in \PL(\PR)$ for any provability predicate $\PR(x)$ of $T$ satisfying $\mathbf{E}$ and $T \vdash \neg (\PR(\gn{\varphi}) \wedge \PR(\gn{\neg \varphi}))$ for any formula $\varphi$. 
%     \item 
%     $A \in \PL(\PR)$ for any $\Sigma_1$ provability predicate $\PR(x)$ of $T$ satisfying $\mathbf{E}$ and $T \vdash \neg (\PR(\gn{\varphi}) \wedge \PR(\gn{\neg \varphi}))$ for any formula $\varphi$. 
% \end{itemize}
% Moreover, there exists a $\Sigma_1$ provability predicate $\PR(x)$ of $T$ such that $\mathsf{END} = \PL(\PR)$.
% \end{cor}

\section{Arithmetical completeness of $\mathsf{ECN}$ and $\mathsf{ECNP}$}\label{pr2}
In this section, we prove the arithmetical completeness theorems for $\mathsf{ECN}$ and $\mathsf{ECNP}$.
\begin{thm}\label{thmECN}
For $L \in \{ \mathsf{ECN}, \mathsf{ENP}  \}$, 
there exists a $\Sigma_1$ provability predicate $\PR_T(x)$ of $T$ satisfying $\mathbf{E}$ such that
\begin{enumerate}
    \item 
    for any $A \in \MF$ and any arithmetical interpretation $f$ based on $\PR_T(x)$, if $L \vdash A$, then $T \vdash f(A)$, and
    \item 
    there exists an arithmetical interpretation $f$ based on $\PR_T(x)$ such that for any $A \in \MF$, $L \vdash A$ if and only if $T \vdash f(A)$.
\end{enumerate}

\end{thm}
\begin{proof}
Let $L \in \{ \mathsf{ECN}, \mathsf{ECNP} \}$.
As in the proof of Theorem \ref{thmEN},
we obtain 
 a primitive recursive enumeration $\langle (W_k, N_k, v_k) \rangle_{k \in \omega}$ of pairwise disjoint finite $L$-models $(W_k, N_k,v_k)$ falsifying $A_k$ for any $k \in \omega$, where  $\langle A_k \rangle_{k \in \omega}$ is an enumeration of all $L$-unprovable formulas.
 Here, we may assume that $\bigcup_{k \in \omega} W_k= \omega \setminus \{0\}$.
We define the primitive recursive function  $g_1$ as follows.
\medskip

\textbf{Procedure 1.}

Stage $s$:
\begin{itemize}
\item If $h(s+1) =0$,
\begin{equation*}
  g_1(s)  = \begin{cases}
       \varphi & \text{if}\ s\ \text{is a}\ T \text{-proof of}\ \varphi, \\
               0 & \text{otherwise}.
             \end{cases}
  \end{equation*}

Then, go to Stage $s+1$.

\item If $h(s+1) \neq 0$, go to Procedure 2.
\end{itemize}

\textbf{Procedure 2.}

Suppose $s$ and $i \neq 0$ satisfy $h(s)=0$ and $h(s+1)=i$. 
Let $k$ be a number such that $i \in W_k$. 
Let 
\[
X = \{ \psi \in F_{s-1} \mid \exists \varphi \in P_{T,s-1} \ \text{such that} \ \varphi \leftrightarrow_{s-1} \psi \}
\]
and 
\begin{align*}
Y = \{ \psi \in F_{s-1} \mid  \exists \varphi & \in P_{T,s-1} \ \exists V \in N_k(i) [ \forall j \in V (\lambda(\num{j}) \to \varphi \in P_{T,s-1}) \\
& \& \  \forall j \in W_k \setminus V (\lambda(\num{j}) \to \neg \varphi \in P_{T,s-1})  \ \& \ \varphi \leftrightarrow_{s-1} \psi \ ]\}.
\end{align*}

For each $n \in \omega$, we define the sequence $\{Z_n\}_{n \in \omega}$ of sets of $\LA$-formulas
 inductively as follows:
\begin{itemize}
\item
$Z_0 :=  P_{T,s-1} \cup X \cup Y$.
\item
$Z_{n+1} := \{ \rho \in F_{s-1} \mid \exists \varphi, \exists \psi \in \bigcup_{j \leq n} Z_j   (\varphi \wedge \psi \leftrightarrow_{s-1} \rho)   \}$.
\end{itemize}
We define the set $Z$ as $\bigcup_{n \in \omega} Z_n$, which is a subset of $F_{s-1}$.
Let $\chi_{0}, \ldots , \chi_{l-1}$ be the enumeration of all formulas in $Z$.
\begin{itemize}
    \item 
For each $t< l$, let $g_1(s+t) = \chi_t$.
    \item
For each $t \geq l$, let $g_1(s+t) = 0$.   
\end{itemize}
The construction of the function $g_1$ has been finished.  We define $\PR_{g_1}(x) \equiv \exists y (g_1(y)=x)$. Note that in Procedure 2, a formula $\varphi$ is output by $g_1$ if and only if $\varphi \in Z$.

The following claim is proved as in the proof of Claim \ref{EN1}.
\begin{cl}\label{ECN1}
Let $\varphi$ be an $\LA$-formula. Then, $\PA + \Con_T \vdash \PR_{g_1}(\gn{\varphi}) \leftrightarrow \Prov_T (\gn{\varphi})$.
\end{cl}
The condition $\mathbf{E}$ for $\PR_{g_1}(x)$ is proved as in the proof of Claim \ref{EN2}.
\begin{cl}\label{ECN2}
For any $\LA$-formulas $\varphi$ and $\psi$, if $T \vdash \varphi \leftrightarrow \psi$, then $\PA \vdash \PR_{g_1}(\gn{\varphi}) \leftrightarrow \PR_{g_1}(\gn{\psi})$ holds.    
\end{cl}
We prove $\PR_{g_1}(x)$ satisfies the condition $\mathbf{C}$.
\begin{cl}\label{ECN3}
For any formulas $\varphi$ and $\psi$, $\PA \vdash \PR_{g_1}(\gn{\varphi}) \wedge \PR_{g_1}(\gn{\psi}) \to \PR_{g_1}(\gn{\varphi \wedge \psi})$ holds.
    \end{cl}
\begin{proof}
Let $\varphi$ and $\psi$ be any $\LA$-formulas.
We obtain $\PA + \Con_T \vdash \PR_{g_1}(\gn{\varphi}) \wedge \PR_{g_1}(\gn{\psi}) \to \PR_{g_1}(\gn{\varphi \wedge \psi})$ by $\mathbf{C}$ for $\Prov_T(x)$ and Claim \ref{ECN1}.

We prove $\PA + \exists x (\lambda(x) \wedge x \neq 0) \vdash \PR_{g_1}(\gn{\varphi}) \wedge \PR_{g_1}(\gn{\psi}) \to \PR_{g_1}(\gn{\varphi \wedge \psi})$. We work in $\PA$: Let $s$, $i$, and $k$ be such that $h(s)=0$ and $h(s+1) = i \in W_k$. Suppose $\PR_{g_1}(\gn{\varphi})$ and $\PR_{g_1}(\gn{\psi})$ hold. Then, we obtain $\varphi, \psi \in Z$, and there exists some $n$ such that $\varphi, \psi \in \bigcup_{j \leq n} Z_{j}$. 
Since $\varphi \wedge \psi \in F_{s-1}$, it follows that
 $\varphi \wedge \psi \in Z_{n+1} \subseteq Z$. Thus $\varphi \wedge \psi$ is output by $g_1$, that is,  $\PR_{g_1}(\gn{\varphi \wedge \psi})$ holds.

Thus, by the law of excluded middle, we obtain $\PA \vdash \PR_{g_1}(\gn{\varphi}) \wedge \ \PR_{g_1}(\gn{\psi}) \to \PR_{g_1}(\gn{\varphi \wedge \psi})$.
\end{proof}

\begin{cl}\label{ECN4}
Let $\rho$ be an $\LA$-formula. 
$\PA$ proves the following:
``Suppose $h(s)=0$ and $h(s+1)=i \in W_k$. For any number $n$, if $\rho \in Z_n$, then 
there exists $V \in N_k(i)$ such that $P_{T,s-1} \tc \lambda(\num{j}) \to \rho$ for any $j \in V$, and  $P_{T,s-1} \tc \lambda(\num{j}) \to \neg \rho$ for any $j \in W_k \setminus V$.''
\end{cl}
\begin{proof}
Let $\rho$ be any formula. We work in $\PA$: Let $h(s)=0$ and $h(s+1)=i \in W_k$.
We prove the statement by  induction on $n$.
We prove the base step $n=0$.
If $n=0$, then $\rho \in Z_0$, that is, $\rho \in P_{T,s-1} \cup X \cup Y$.
If $\rho \in P_{T,s-1} \cup X$, then $P_{T,s-1} \tc \rho$.
Since $(W_k,N_k)$ is an $\EN$-frame, we obtain $W_k \in N_k(i)$ and $P_{T,s-1} \tc \lambda(\num{j}) \to \rho$ for any $j \in W_k$.
If $\rho \in Y$, then by the definition of $Y$, there exists a set $V \in N_k (\num{i})$ such that  $P_{T,s-1} \tc \lambda(\num{j}) \to \rho$ for any $j \in V$ and  $P_{T,s-1} \tc \lambda(\num{j}) \to \neg \rho$ for any $j \in W_k \setminus V$.

Next, we prove the induction step $n+1$. Suppose $\rho \in Z_{n+1}$. Then, there exist $\varphi, \psi \in \bigcup_{j \leq n}Z_j$ such that $\varphi \land \psi \leftrightarrow_{s-1} \rho$. Since $\varphi, \psi \in \bigcup_{j \leq n}Z_j$, by the induction hypothesis, there exist $V_0, V_1 \in N_k(i)$ such that
% \begin{enumerate}
%     \item[(i)]
% $\varphi$ and $\psi$ are t.c.'s of $P_{T,s-1}$.
% In this case, $\varphi \wedge \psi$ is a t.c.~of $P_{T,s-1}$.
% Thus, by $\rho \leftrightarrow_{s-1} \varphi \wedge \psi$, it follows that $\rho$ is a t.c.~of $P_{T,s-1}$.
% \item[(ii)] 
% $\varphi$ is a t.c.~of $P_{T,s-1}$ and  $\exists V \in N_k(i) \ \bigl( \forall j \in V ( P_{T,s-1} \tc \lambda(\num{j}) \to \psi$) \& \\ $\forall j \in W_k \setminus V( P_{T,s-1} \tc \lambda(\num{j}) \to \neg \psi)\bigr)$.
% Since $\rho \leftrightarrow_{s-1} \varphi \wedge \psi$  and $P_{T,s-1} \tc \varphi$ hold, we obtain 
% \[
% P_{T,s-1} \tc \rho \leftrightarrow \psi.
% \]
% Therefore, it follows that for any $j \in V$,
% \[
% P_{T,s-1} \tc \lambda(\num{j}) \to \rho,
% \]
% and for any $j \in W_k \setminus V$,
% \[
% P_{T,s-1} \tc \lambda(\num{j}) \to \neg \rho.
% \]
%     \item[(iii)] 
% $\psi$ is a t.c.~of $P_{T,s-1}$ and  $\exists V \in N_k(i) \ \bigl( \forall j \in V ( P_{T,s-1} \tc \lambda(\num{j}) \to \varphi$) \& \\ $\forall j \in W_k \setminus V( P_{T,s-1} \tc \lambda(\num{j}) \to \neg \varphi)\bigr)$.
% This case is proved in the same way as (ii).
\begin{itemize}
    \item 
\(\forall j \in V_0 \ ( P_{T,s-1} \tc \lambda(\num{j}) \to \psi )\) \&  
\(\forall j \in W_k \setminus V_0 \ ( P_{T,s-1} \tc \lambda(\num{j}) \to \neg \psi )\), and 
    \item 
\(\forall j \in V_1 \ ( P_{T,s-1} \tc \lambda(\num{j}) \to \varphi )\) \&  
\(\forall j \in W_k \setminus V_1 \ ( P_{T,s-1} \tc \lambda(\num{j}) \to \neg \varphi )\).
\end{itemize}

Since $V_0, V_1 \in N_k(i)$, the frame condition of $\mathsf{ECN}$ ensures that  $V_0 \cap V_1 \in N_k(i)$. Also, for any $j \in V_0 \cap V_1$,
\[
P_{T,s-1} \tc \lambda(\num{j}) \to \varphi \wedge \psi,
\]
which implies 
\[
P_{T,s-1} \tc \lambda(\num{j}) \to \rho,
\]
by $\rho \leftrightarrow_{s-1} \varphi \wedge \psi$. On the other hand, for any $j \in W_{k} \setminus (V_0 \cap V_1)$, that is, $j \in (W_k \setminus V_0) \cup (W_k \setminus V_1)$, we obtain 
\[
P_{T,s-1} \tc \lambda(\num{j}) \to \neg \varphi \vee \neg \psi,
\]
which implies 
\[
P_{T,s-1} \tc \lambda(\num{j}) \to \neg \rho.
\]
\qedhere
\end{proof}

\begin{cl}\label{ECNP}
If $L$ is $\mathsf{ECNP}$, then $\PA \vdash \neg \PR_{g_1}(\gn{0=1})$ holds.
    
\end{cl}
\begin{proof}
Since $\PA + \Con_T \vdash \neg \PR_{g_1}(\gn{0=1})$ by Claim \ref{ECN1}, it suffices to
 prove $\PA + \exists x (\lambda(x) \wedge x \neq 0) \vdash \neg \PR_{g_1}(\gn{0=1})$. We work in $\PA + \exists x (\lambda(x) \wedge x \neq 0)$: Suppose $h(s)=0$ and $h(s+1)= i \in W_k$.
Suppose, towards a contradiction, that 
$\PR_{g_1}(\gn{0=1})$ holds, that is, $0=1$ is output by $g_1$. Then, it follows that $0=1 \in Z$, which implies that there exists a number $n$ such that $0=1 \in Z_n$.
Thus, by Claim \ref{ECN4}, 
there exists $V \in N_k(i)$ such that  $P_{T,s-1} \tc \lambda(\num{j}) \to 0=1$ for any $j \in V$ and  $P_{T,s-1} \tc \lambda(\num{j}) \to 0 \neq 1$ for any $j \in W_k \setminus V$.
By the frame condition of $\mathsf{ENP}$, we obtain $\emptyset \notin N_k(i)$.
Then, there exists $j \in V$ such that $P_{T,s-1} \tc \lambda(\num{j}) \to 0=1$.
Since 
$0 \neq 1$ has a standard $T$-proof,
$P_{T,s-1} \tc 0 \neq 1$ holds.
Therefore, we obtain $P_{T,s-1} \tc \neg \lambda(\num{j})$.
This contradicts $h(s)=0$.

Thus, we obtain $\PA + \neg \Con_T \vdash \neg \PR_{g_1}(\gn{0=1})$.
By the law of excluded middle, it follows that $\PA \vdash \neg \PR_{g_1}(\gn{0=1})$.
\end{proof}
We define the arithmetical interpretation $f_1$ by
$f_1(p) \equiv \exists x \exists y( x \in W_y \wedge \lambda(x) \wedge x \neq 0 \wedge x \in v_y(p))$.
\begin{cl}\label{ECN5}
Let $B \in \MF$ and $i \in W_k$.
\begin{enumerate}
\item
If $i \in v_k(B)$, then $\PA \vdash \lambda(\num{i}) \to f_1(B)$.    
\item 
If $i \notin v_k(B)$, then $\PA \vdash \lambda(\num{i}) \to \neg f_1(B)$.
\end{enumerate}   
\end{cl}
\begin{proof}
We prove only the case $B \equiv \Box C$.
 
1. This case is proved in the same way as in (i) of the proof of Claim \ref{cl 5.15}.
 
2. Suppose $i \notin v_k(\Box C)$. 
We work in $\PA + \lambda(\num{i})$:
Let $s$ be such that $h(s)=0$ and $h(s+1)=i \in W_k$.
We prove that $f_1(C)$ is not output by $g_1$, that is, $f_1(C) \notin Z$.
Suppose, towards a contradiction, that $f_1(C) \in Z$. Then, there exists a number $n$ such that $f_1(C) \in Z_n$.
By Claim \ref{ECN4}, 
    there exists $V \in N_k(i)$ such that for any $j \in V$, $P_{T,s-1} \tc \lambda(\num{j}) \to f_1(C)$ and for any $j \in W_k \setminus V$, $P_{T,s-1} \tc \lambda(\num{j}) \to \neg f_1(C)$.
 In the same way as in the case (ii) of item (2) of Claim \ref{cl 5.15}, we obtain a contradiction.
Thus, it follows that $f_1(C) \notin Z$, and hence $f_1(C)$ is not output by $g_1$.
\qedhere    
\end{proof}
We finish our proof of Theorem \ref{thmECN}. The implication $(\Rightarrow)$ of Theorems \ref{thmECN}.1 and \ref{thmECN}.2 follows from Claims \ref{ECN2}, \ref{ECN3}, and \ref{ECNP}.
As in the proof of Theorem \ref{thmEN}, we obtain the implication $(\Leftarrow)$ of Theorems \ref{thmECN}.1 and  \ref{thmECN}.2 by Claim \ref{ECN5}.
\end{proof}

\begin{cor}[The arithmetical completeness of $\mathsf{ECN}$]
\begin{align*}
    \ECN & = \bigcap \{ \PL(\PR_T) \mid \PR_T(x) \text{ is a provability predicate satisfying } \mathbf{E} \text{ and } \mathbf{C} \}\\
    & =  \bigcap \{ \PL(\PR_T) \mid \PR_T(x) \text{ is a } \Sigma_1 \text{ provability predicate satisfying } \mathbf{E} \text{ and } \mathbf{C}  \}.
\end{align*}
Moreover, there exists a $\Sigma_1$ provability predicate $\PR_T(x)$ of $T$ such that $\mathsf{ECN} = \PL(\PR_T)$.
\end{cor}

\begin{cor}[The arithmetical completeness of $\mathsf{ECNP}$]
\begin{align*}
    \ECNP & = \bigcap \{ \PL(\PR_T) \mid \PR_T(x) \text{ satisfies } \mathbf{E},  \mathbf{C}, \text{ and } T \vdash \Con^L_T \}\\
    & = \bigcap \{ \PL(\PR_T) \mid \PR_T(x) \text{ is } \Sigma_1 \text{ and satisfies } \mathbf{E}, \mathbf{C} \text{ and } T \vdash \Con^L_T \}.
\end{align*}
Moreover, there exists a $\Sigma_1$ provability predicate $\PR_T(x)$ of $T$ such that $\mathsf{ECNP} = \PL(\PR_T)$.
\end{cor}

\section{Concluding remarks}\label{fur}
In this paper, we have investigated the condition $\mathbf{E}$ from a modal logical perspective.
In particular, we have proved the arithmetical completeness theorems for the logics $\EN$, $\ECN$, $\ENP$, $\END$, and $\ECNP$ by embedding neighborhood models into arithmetic.
Although we do not discuss the details in this paper, we remark here that for extensions of $\mathsf{EN}$ such as $\mathsf{MN}$, $\mathsf{K}$, and $\mathsf{KD}$,
whose arithmetical completeness has been established via relational semantics \cite{KK2,Kur18-2,Kur23,Kur20}, our approach based on neighborhood semantics can also be adapted to obtain alternative proofs of arithmetical completeness.

Since the conditions $\D{3}$ and $\mathbf{E}$ together yield the second incompleteness theorem for $\Con^S_T$, their combination is also a subject worth analyzing from the perspective of provability logic.
The extensions $\mathsf{EN4}$, $\mathsf{ENP4}$, and $\mathsf{ECN4}$ of $\EN$ are defined as follows:
\begin{itemize}
    \item $\mathsf{EN4}= \EN + (\Box A \to \Box \Box A)$,
    \item $\mathsf{ENP4}= \mathsf{EN4} + \neg \Box \bot$,
    \item $\mathsf{ECN4}= \mathsf{ECN} + (\Box A \to \Box \Box A)$.
\end{itemize}
The finite frame property of $\mathsf{EN4}$ and $\mathsf{ENP4}$ was proved in \cite{Kop}.
However, establishing the arithmetical completeness of these logics, as well as for $\mathsf{ECN4}$, involves technical difficulties. In particular, whether $\mathsf{EN4}$ is arithmetically complete or arithmetically incomplete remains unclear, that is, for a provability predicate $\PR_T(x)$ satisfying $\mathbf{E}$ and $\mathbf{D3}$, it is uncertain whether the fixed-point theorem produces additional modal principles over arithmetic. It is an interesting aspect of the present work that such technically challenging questions naturally arise.
\begin{prob}
For each $L \in \{ \mathsf{EN4}, \mathsf{ENP4}, \mathsf{ECN4}  \}$,  does there exist a provability predicate $\PR_T(x)$ such that  $\PL(\PR_T)=L$ ?  \end{prob}
Moreover, for each $L \in \{\mathsf{EN4}, \mathsf{ENP4}, \mathsf{ECN4} \}$, it remains open whether there exists a provability predicate $\PR_T(x)$ exactly corresponding to $L$, that is, we propose the following problem:
\begin{prob}
\leavevmode
\begin{itemize}
    \item Does there exist a provability predicate $\PR_T(x)$ such that $\PR_T(x)$ satisfies $\mathbf{E}$ and $\D{3}$, but not $\mathbf{C}$, $\mathbf{M}$, and $T \vdash \Con^L_T$.
    \item 
    Does there exist a provability predicate $\PR_T(x)$ such that $\PR_T(x)$ satisfies $\mathbf{E}$, $\D{3}$, and $T \vdash \Con^L_T$, but not $\mathbf{C}$ and $\mathbf{M}$.
    \item Does there exist a provability predicate $\PR_T(x)$ such that $\PR_T(x)$ satisfies $\mathbf{E}$, $\mathbf{C}$, and $\D{3}$, but not $\mathbf{M}$ and $T \vdash \Con^L_T$.
\end{itemize}
\end{prob}
This paper is part of a research project on the modal logical analysis of derivability conditions. Within this project, including the results of the present work, a number of results have been accumulated. The author and Kurahashi provided in \cite{KK3} an overview of these results and the remaining open problems. 
% Also, an overview of modal logical studies on derivability conditions and consistency statements is provided in \cite{KK3}.

\section*{Acknowledgments}
This work was supported by JST SPRING, Grant Number JPMJSP2148.
The author would like to thank Taishi Kurahashi for many valuable discussions and comments. 
\bibliographystyle{plain}
\bibliography{refs}

@article{Kur25,
 author               = {Kurahashi, Taishi},
 note                 = {arXiv:2507.00955},
 title                = {Refinements of provability and consistency principles for the second incompleteness theorem},
 year                 = {2025},
 }

@article{Kop,
  author    = {Kirill Kopnev},
  title     = {The Finite Model Property of Some Non-normal Modal Logics with the Transitivity Axiom},

  note   = {arXiv:2305.08605},
  year      = {2023},
  url       = {https://arxiv.org/abs/2305.08605}
}

@article{KK3,
    author = {Kogure, Haruka and Kurahashi, Taishi},
    title = {Modal logical aspects of provability predicates and consistency statements},
    note={arXiv:2511.15531},

}

@article{Kur20-2
,
	author = {Taishi Kurahashi},
	doi = {10.1017/jsl.2020.33},
	journal = {Journal of Symbolic Logic},
	number = {3},
	pages = {1224--1253},
	title = {A Note on Derivability Conditions},
	volume = {85},
	year = {2020}
}

@article{lew,
	author = {David K. Lewis},
	doi = {10.1007/bf00257488},
	journal = {Journal of Philosophical Logic},
	number = {4},
	pages = {457--466},
	publisher = {Springer},
	title = {Intensional Logics Without Interative Axioms},
	volume = {3},
	year = {1974}
}

@book{Eric,
	address = {New York},
	author = {Eric Pacuit},
	editor = {},
SERIES ={Short Textbooks in Logic},
	publisher = {Springer Cham},
	title = {Neighborhood Semantics for Modal Logic},
	year = {2017}
}

@book{Che,
	address = {New York},
	author = {Brian F. Chellas},
	editor = {},
	publisher = {Cambridge University Press},
	title = {Modal Logic: An Introduction},
	year = {1980}
}

@article {KK2,
    AUTHOR = {Kogure, Haruka and Kurahashi, Taishi},
     TITLE = {Arithmetical completeness theorems for monotonic modal logics},
  JOURNAL = {Annals of Pure and Applied Logic},
    VOLUME = {174},
      YEAR = {2023},
    NUMBER = {7},
     PAGES = {Paper No. 103271},
      ISSN = {0168-0072,1873-2461},
   MRCLASS = {03F45 (03B45 03F40)},
  MRNUMBER = {4578069},
       DOI = {10.1016/j.apal.2023.103271},
       URL = {https://doi.org/10.1016/j.apal.2023.103271},
}

@article{fmt,
  title    = {The pure logic of necessitation},
  author   = {Fitting, {Melvin C.} and Marek, {V. Wiktor} and Miroslaw Truszczy{\'n}ski},
  year     = {1992},
  doi      = {10.1093/logcom/2.3.349},
  language = {English},
  journal = {Journal of Logic and Computation},
  volume   = {2},
  pages    = {349-373},
  number   = {3}
}

@article{Kur23,
  author       = {Taishi Kurahashi},
  title        = {The provability logic of all provability predicates},
  journal      = {Journal of Logic and Computation},
  volume       = {34},
  number       = {6},
  pages        = {1108--1135},
  year         = {2024},
  doi          = {10.1093/logcom/exad060},
}

@Article{kog,
 Author = {Kogure, Haruka},
 Title = {Arithmetical completeness for some extensions of the pure logic of necessitation},
 Note = {arXiv:2409.00938},
 Year = {2024}
}

@article {Kur18-2,
    AUTHOR = {Kurahashi, Taishi},
     TITLE = {Arithmetical completeness theorem for modal logic {$\mathsf{K}$}},
   JOURNAL = {Studia Logica},
  FJOURNAL = {Studia Logica. An International Journal for Symbolic Logic},
    VOLUME = {106},
      YEAR = {2018},
    NUMBER = {2},
     PAGES = {219--235},
      ISSN = {0039-3215,1572-8730},
   MRCLASS = {03F45 (03B45)},
  MRNUMBER = {3777441},
MRREVIEWER = {Ming\ Hsiung},
       DOI = {10.1007/s11225-017-9735-y},
       URL = {https://doi.org/10.1007/s11225-017-9735-y},
}

@Article{Kur20,
 Author = {Kurahashi, Taishi},
 Title = {Rosser provability and normal modal logics},
 Journal = {Studia Logica},
 ISSN = {0039-3215},
 Volume = {108},
 Number = {3},
 Pages = {597--617},
 Year = {2020},
 Language = {English},
 DOI = {10.1007/s11225-019-09865-2},
 zbMATH = {7210302},
 Zbl = {1481.03007}
}

@article{Sol,
  author  = {Robert M. Solovay},
  title   = {Provability Interpretations of Modal Logic},
  journal = {Israel Journal of Mathematics},
  volume  = {25},
  pages   = {287--304},
  year    = {1976},
}

% \section*{Appendix: Provability predicate corresponding to  $\mathsf{ENP4}$}
% We construct a provability predicate corresponding satisfies $\mathbf{E}$, $\D{3}$, and $T \vdash \Con^L_T$, but fails $\mathbf{C}$ and $\mathbf{M}$.

\end{document}